\documentclass{amsart}
\usepackage{tikz}
\usepackage{graphicx}

\usepackage{amssymb}
\usepackage{amsmath}
\usepackage{amsthm, amsfonts, mathrsfs}

\newtheorem{theorem}{Theorem}
\newtheorem{lemma}[theorem]{Lemma}
\newtheorem{cor}[theorem]{Corollary}
\newtheorem{prop}[theorem]{Proposition}

\newtheorem{ques}{Question}

\newtheorem{exam}[theorem]{Example}

\title{Homotopy types of  Vietoris-Rips complexes of Hypercube  Graphs}

\date{May 2023}

\author[Z. Feng]{Ziqin Feng}
\address{Department of Mathematics and Statistics\\Auburn University\\Auburn, AL~36849}
\email{zzf0006@auburn.edu}

\date{\today}
\subjclass[2020]{05E45, 55P10, 55N31, 55U10}
\keywords{Vietories-Rips Complexes, Simplicial Complexes, Homotopy Types,  Hypercube Graphs}

\begin{document}
\begin{abstract} We describe the homotopy types of Vietoris-Rips complexes of hypercube graphs at scale $3$. We represent the vertices in the hypercube graph $Q_m$ as the collection of all subsets of $[m]=\{1, 2, \ldots, m\}$ and equip $Q_m$ with the  metric using symmetric difference distance. It is proved in \cite{AA22} that the Vietoris-Rips complexes of hypercube graphs $Q_m$ at scale $2$, $\mathcal{VR}(Q_m; 2)$, is homotopy equivalent to $c_m$-many spheres with dimension $3$ where $c_m=\sum_{0\leq j< i<m} (j+1)(2^{m-2}-2^{i-1})$. Questions are raised in \cite{AA22} for determining the homotopy types of $\mathcal{VR}(Q_m, r)$ with large scales $r=3, 4, \ldots, m-2$. We prove that for $m\geq 5$, $$\mathcal{VR}(Q_m; 3)\simeq (\bigvee_{2^{m-4}\cdot{m\choose 4}} S^7) \vee (\bigvee_{\sum_{i=4}^{m-1}2^{i-4}\cdot{i\choose 4}} S^4).$$

\end{abstract}
\maketitle

\section{Introduction}

The Vietoris-Rips complex $\mathcal{VR}(X;r)$ of a metric space $(X,d)$ with a scale parameter $r\geq 0$, introduced in \cite{VI27} and \cite{EG87}, is a simplicial complex with vertex set $X$, where $\sigma\in [X]^{<\infty}$ is a simplex in $\mathcal{VR}(X;r)$ if and only if its diameter $\text{diam}(\sigma)\leq r$. Here, $[X]^{<\infty}$ denotes the collection of all finite subsets of $X$, and for any subset $S$ of $X$, $\text{diam}(S)$ is defined as the supremum of all distances $d(x,y)$ between pairs of points $x,y\in S$. It is important to determine the homotopy types of Vietoris-Rips complex of finite metric spaces in applied topology due to the development of the application of persistent homology in data analysis \cite{GC09,RG08, ZC05}. Recently, a lot of attention has been drawn to study the homotopy types of Vietoris-Rips complexes of different metric spaces, for example, circles and ellipses (\cite{AA17, AAR19}), metric graphs (\cite{GGPSWWZ18}), geodesic spaces (\cite{ZV19, ZV20}), planar point sets (\cite{ACJS18}, \cite{CDEG10}), Hypercube graphs (\cite{AA22, FN23, Shu22}), and more (\cite{MA13}, \cite{AAGGPS20}, \cite{AFK17}).

In this paper, we investigate the homotopy type of the Vietoris-Rips complex $\mathcal{VR}(Q_m; 3)$ where $Q_m$ is the hypercube graph. The hypercube graph is a graph whose vertices are all binary strings of length $m$, denoted by $Q_m$, and whose edges are given by pairs of such strings with Hamming distance $1$. The Hamming distance between any two binary strings with the same length is defined as the number of positions in which their entries differ. Therefore, $Q_m$ is a metric space and the hypercube graph is the Vietoris-Rips complex $\mathcal{VR}(Q_m; 1)$. Each binary string of length $m$ can also be represented as a subset of $[m]=\{1, 2, \ldots, m\}$, so another representation of $Q_m$ is the power set of $[m]$ equipped with the symmetric difference distance. For any $A$ and $B$ being subsets of  $[m]$, the distance between $A$ and $B$ is defined as $d(A, B)=|A\Delta B|$, where $A\Delta B$ denotes the symmetric difference of $A$ and $B$, i.e., $(A\setminus B)\cup (B\setminus A)$, and $|A\Delta B|$ denotes the cardinality of $A\Delta B$. In this paper, we'll use subsets of $[m]$ to represent the vertices in $Q_m$.

Adamaszek and Adams investigated the  complexes $\mathcal{VR}(Q_m; r)$ with small scales $r=0,1,2$ in \cite{AA22}. The complex $\mathcal{VR}(Q_m; 0)$ is a wedge sum of $(2^m-1)$-many $S^0$'s, and $\mathcal{VR}(Q_m; 1)$ is a wedge sum of $((m-2)2^{m-1}+1)$-many $S^1$'s. Their main result in \cite{AA22} is that the complex $\mathcal{VR}(Q_m; 2)$ is homotopy equivalent to a wedge sum of $c_m$ copies of $S^3$'s, where $c_m$ is given by $c_m=\sum_{0\leq j<i<m}(j+1)(2^{m-2}-2^{i-1})$. Note that the complex  $\mathcal{VR}(Q_m; m-1)$ is a boundary of the $(2^{m-1})$-dimensional cross-polytope with $2^m$ vertices and therefore homotopy equivalent to a sphere with dimension $2^{m-1}-1$. Then it is a natural question, as the authors raised in \cite{AA22}, to ask  what is the homotopy types of $\mathcal{VR}(Q_m; r)$ with large scales $r=3, 4, \ldots, m-2$. With the help of Polymake \cite{Poly10} and Risper++ \cite{Zhang}, Adamaszek and Adams in \cite{AA22} also computed the reduced homology groups of $\mathcal{VR}(Q_m; 3)$ for $m=5, 6, \ldots, 9$, with coefficients $\mathbb{Z}$ or $\mathbb{Z}/2\mathbb{Z}$. The results show that these homology groups are nontrivial only in dimensions $4$ and $7$, indicating that the complex $\mathcal{VR}(Q_m; 3)$ is homotopy equivalent to  a wedge sum of copies of $S^4$'s and $S^7$'s. These computations suggest that the complexes $\mathcal{VR}(Q_m; 3)$ for $m\geq 5$ are homotopy equivalent to a wedge sum of spheres with different dimensions which are more complicated than that of the complexes $\mathcal{VR}(Q_m; r)$ with $r=0, 1, 2$. Shukla \cite{Shu22} subsequently confirmed the computational indications by proving that for $m\geq 5$, the reduced homology group $\tilde{H}_i(\mathcal{VR}(Q_m; 3))$ is nontrivial if and only if $i\in \{4, 7\}$.  We acquire the homotopy types of $\mathcal{VR}(Q_m; 3)$ by  proving that,  in Theorem~\ref{main}, for $m\geq 5$, $$\mathcal{VR}(Q_m; 3)\simeq (\bigvee_{2^{m-4}\cdot{m\choose 4}} S^7) \vee (\bigvee_{\sum_{i=4}^{m-1}2^{i-4}\cdot{i\choose 4}} S^4).$$

To describe the homotopy types of some Vietoris-Rips complexes, the authors in \cite{FN23} carried out induction on a total order '$\prec$' in $Q_m$ defined as following. For any $A, B\subseteq  [m]$,  we say $A\prec B$ if one of the followings holds:
\begin{itemize}
\item[i)] $|A|< |B|$;

\item[ii)] there is a $k\in\mathbb{N}$ such that $i_k<j_k$ and $i_{\ell}=j_{\ell}$ for any $\ell<k$, when $n=|A|=|B|$, $A=\{i_1, i_2, \ldots, i_n\}$ and $B=\{j_1,j_2, \ldots, j_n\}$ with $i_1<i_2<\ldots <i_n$ and $j_1<j_2<\ldots<j_n$.

\end{itemize}

For $n\leq m$,  we denote $\mathcal{F}_{n}^m$ be the collection of all subsets of $[m]$ with cardinality $n$. As mentioned in \cite{FN23}, the complex $\mathcal{VR}(\mathcal{F}_{n}^m; r)$ is closely related to the independence complex of Kneser graphs. Barmak in \cite{Bar13} showed how to use independence complexes to investigate the colorability of graphs, and in fact, he obtain lowers bounds for the chromatic number of a graph using a numerical homotopy invairant associated to its independence complex.  It is proved in \cite{FN23} that the complex $\mathcal{VR}(\mathcal{F}_n^m; 2)$ and $\mathcal{VR}(\mathcal{F}_n^m\cup \mathcal{F}_{n+1}^m; 2)$ are either contractible or  homotopy equivalent to a wedge sum of $S^2$'s; $\mathcal{VR}(\mathcal{F}_n^m\cup \mathcal{F}_{n+2}^m; 2)$ are either contractible or  homotopy equivalent to a wedge sum of $S^3$'s.    For any $A\subseteq [m]$, we denote $\mathcal{F}_{\prec A}=\{B: B\prec A \text{ and }B\subseteq [m]\}$ and $\mathcal{F}_{\preceq A}^m=\mathcal{F}_{\prec A}^m\cup \{A\}$. Also for the purpose of convenience, we represent $A=\{i_1, i_2, \ldots, i_n\}\subseteq [m]$ as $i_1i_2\cdots i_n$  when $i_1<i_2<\ldots<i_n$. For any $a=1, 2, \ldots, n$, we denote $i_1i_2\cdots \hat{i_a}\cdots i_n =i_1i_2\cdots i_n\setminus \{i_a\}$.  As an example, in the following contexts, $A=i_1$ means $A$ is a set consisting the number $i_1$, i.e. $A=\{i_1\}$. For concrete vertices, for example $A=\{2, 3, 4\}$, we still use the original notation to avoid confusions.  It is straightforward to see that the complexes $\mathcal{VR}(\mathcal{F}_{\leq 3}^m; 3)$ are cones, hence contractible, here $\mathcal{F}_{\leq 3}^m = \mathcal{F}_{0}^m\cup \mathcal{F}_{1}^m\cup \mathcal{F}_{2}^m \cup  \mathcal{F}_{3}^m$.

\medskip

The paper is organized in the following way. In Section~\ref{sc_sc}, we develop several techniques which are useful to study the homotopy types of Vietoris-Rips complexes. One of the techniques is that we establish a sufficient condition in Lemma~\ref{sc_homo_t} which guarantees that SC$_K(L)\simeq L$ for $L$ being a subcomplex of $K$, here SC$_K(L)=\bigcup_{v\in L}\text{st}_K(v)$. This gives a generalization of the conditions given in \cite{Bar13, FN23} and is heavily used in the following context.  In Section~\ref{mx_f1f2f3_sec}, we prove that the complex $\mathcal{VR}(\mathcal{F}_{1}^n\cup \mathcal{F}_{2}^n \cup  \mathcal{F}_{3}^n; 3)$ is homotopy equivalent to a wedge sum of $S^5$'s for $n\geq 4$. This is used to establish the homotopy type of the link of a vertex $A$ with $|A|=n\geq 4$ in the complex $\mathcal{VR}(\mathcal{F}_{\preceq A}^m; 3)$ (Lemma~\ref{lk_A_preceq}) which is a key discussion in Section~\ref{lk_A_sec}. The link of a vertex $A$ with $|A|=n\geq 4$ in the complex $\mathcal{VR}(\mathcal{F}_{\preceq A}^m; 3)$ is homotopy equivalent to a wedge sum of $S^6$'s and/or $S^3$'s. Then in Section~\ref{main_sec}, we use the the link of a vertex $A$  in  $\mathcal{VR}(\mathcal{F}_{\preceq A}^m; 3)$ to carry an induction on the total order $\prec$ on $Q_m$ which proves the main result (Theorem~\ref{main}). We conclude the paper with some open questions in Section~\ref{conclusion_sec}.

\section{Simpicial Complexes and Star clusters}\label{sc_sc}

\textbf{Topological Spaces and Wedge sums.} For any two topological spaces $X$ and $Y$, we write $X\simeq Y$ when they are homotopy equivalent. The wedge sum of $X$ and $Y$, denoted by $X\vee Y$, is the space obtained by gluing $X$ and $Y$ together at a single point. The homotopy type of $X\vee Y$ is independent of the choice of points if $X$ and $Y$ are connected CW complexes. For $k\ge 1$, $\vee_k X$ denotes the $k$-fold wedge sum of $X$.  We denote $\Sigma X$ to be the suspension of $X$ for any topological space $X$.   For any $k$-dimensional sphere $S^k$, $\Sigma S^k$ is homeomorphic to $S^{k+1}$. A function $f$ from $X$ to $Y$ is said to be null-homotopic if it is homotopic to a constant map.  And it is well-known that any mapping from $S^n$ to $S^m$ is null-homotopic when $n<m$. Hence inductively any mapping from a wedge sum of  $S^n$'s  to  a wedge sum of $S^m$'s is null-homotopic when $n<m$.

Any two metric spaces $(X, d_X)$ and $(Y, d_Y)$ is said to be isometric if there is a bijective distance-preserving map $f$ from $X$ to $Y$, i.e., $d_X(x_1, x_2)=d_Y(f(x_1), f(x_2))$ for any $x_1, x_2\in X$. Hence if $X$ and $Y$ are isometric, then it is straightforward to verify that  $\mathcal{VR}(X; r)$ is homeomorphic to $\mathcal{VR}(Y; r)$ for any $r\geq 0$.

\medskip

\textbf{Simplicial complexes.} A simplicial complex $K$ on a vertex set $V$ is a collection of subsets of $V$ such that: i) all singletons are in $K$; and ii) if $\sigma\in K$ and $\tau\subset \sigma$, then $\tau\in K$. For a complex $K$, we use $K^{(k)}$ to represent the $k$-skeleton of $K$ which is a subcomplex of $K$.  Therefore, we can use $K^{(0)}$ to represent the vertex set $V$.

For vertices $v_1, v_2, \ldots, v_{k+1}$ in a complex $K$,  if they span a simplex in $K$, then we denote the simplex to be $\{v_1, v_2, \ldots, v_{k+1}\}$ which is a $k$-simplex with boundary homotopy equivalent to $S^{k-1}$.   If $\sigma$ and $\tau$ are simplices in $K$ with $\sigma\subset \tau$, we say $\sigma$ is a face of $\tau$. And $\sigma$ is a proper face of $\tau$ if it is a face of $\tau$ and $\sigma\neq \tau$.    We say a simplex is a maximal simplex (or a facet) if it is not a face of any other simplex. For any complex $K$, we denote $M(K)$ to be the collection of maximal faces in $K$. Also for each simplex $\sigma$, we denote $K_\sigma$ to be the complex generated by $\sigma$, i.e. $K_\sigma$ contains $\sigma$ and all its faces.

We say that $L$ is a full subcomplex $K$ if it contains all the simplicies in $K$ spanned by the vertices in $L$, or we say $L$ is the induced complex set on the vertex set $L^{(0)}$.


A complex $K$ is \emph{clique} if  $\sigma\in K$ for each non-empty subset of vertices $\sigma$ with $\{v, w\}\in K$ for any $v, w\in \sigma$.
Note that the Vietoris-Rips complex over any metric space is clique by the definition. Also, if $Y$ is a subspace of a metric space $X$ and $r\geq 0$, then the complex $\mathcal{VR}(Y; r)$ is a clique full subcomplex of $\mathcal{VR}(X; r)$.

The join of two (disjoint) complexes $K$ and $K'$ is the complex $$K\ast K'=K\cup K' \cup \{\sigma\cup \sigma': \sigma\in K\text{ and }\sigma'\in K'\}.$$
The join of a contractible complex with another complex is always contractible. For simplicity, we will identify a simplicial complex with its geometric realization. More information about simplical complexes can be found in \cite{Hatcher, Munkres}

\medskip

The following result is an important technique to investigate the homotopy type of a complex by splitting it into two or more subcomplexes which is proved in \cite{GSS22}.

\begin{lemma}\label{cup_simp} The simplicial complex $K=K_1\cup K_2$ satisfies that the inclusion maps $\imath_1: K_1\cap K_2\rightarrow K_1$ and $\imath_2: K_1\cap K_2\rightarrow K_2$ are both null-homotopic. Then $$K\simeq K_1\vee K_2\vee\Sigma(K_1\cap K_2).$$
\end{lemma}

The star of a vertex $v$ in $K$ is st$_K(v)=\{\sigma: \sigma\cup \{v\}\in K\}$. Hence for any $v\in V$, st$_K(v)$ is contractible because it is a cone with the vertex $v$, namely $v\ast \text{lk}_K(v)$ where $\text{lk}_K(v)=\{\sigma: \sigma\cup \{v\}\in K\text{ and }v\notin \sigma\}$.  An easy corollary of Lemma~\ref{cup_simp} is that $K =K\setminus v\cup \text{st}_K(v) \simeq K\setminus v\vee (K\setminus v\cap \text{st}_K(v))$. Note that $K\setminus v\cap \text{st}_K(v)=\text{lk}_K(v)$. So we get the next lemma (see \cite{AA22}, Lemma 1). Here for any subcomplex $L$ of $K$, $K\setminus L$ denotes the induced complex on the vertext set $K^{(0)}\setminus L^{(0)}$; specifically, for any vertex $v\in K$, $K\setminus v$ is the induced complex on the vertex set $K^{(0)}\setminus \{v\}$.

\begin{lemma} \label{complex_add_1v} If $v$ is a vertex in $K$ with the inclusion map $\imath: \text{lk}_K(v)\rightarrow K$ being null-homotopic, then $K$ is homotopic to $K\setminus v\vee \Sigma (\text{lk}_K(v))$.

\end{lemma}


Hence it is important to decide the inclusion map of a subcomplex $L$ of $K$ is null-homotopic. In general, the mapping from $S^m$ to $S^n$ needs not to be null-homotopic when $m>n$. The next lemma provides a tool to show such inclusion map is null-homotopic and the proof is straightforward.
\begin{lemma} \label{null_homo} Let $L$ be a subcomplex of $K$. The inclusion map $\imath: L\rightarrow K$ is null-homotopic if there is a subcomplex $L_1$ of $K$ such that $L\subseteq L_1$ and the inclusion map $\imath_1: L_1\rightarrow K$ is null-homotopic. \end{lemma}

Barmak in \cite{Bar13} introduced a useful concept, the star cluster of a simplex, to understand the homotopy type of  some complexes. He showed that the star cluster of a simplex $\sigma$ in $K$, SC$_K(\sigma)$, is contractible, i.e. homotopy equivalent to $\sigma$; and he used this tool to answer a question arisen from works of Engstr\"{o}m and Jonsson and investigated lots of examples appearing from literature about independence complex of graphs. The authors in \cite{FN23} generalized this concept to any subcomplex $L$ of $K$ by defining  the \emph{star cluster} of $L$ in $K$ as the subcomplex  $$\text{SC}_K(L)=\bigcup_{v\in L} \text{st}_K(v).$$

In general, SC$_K(L)$ is not homotopy equivalent ot $L$. Example can be found in \cite{FN23}. It is natural to ask for the conditions under which $\text{SC}_K(L)$ is homotopy equivalent to $L$.  In \cite{FN23}, the authors proved that $\text{SC}_K(L)
\simeq L$ if $L$ is a clique subcomplex of a clique complex $K$ such that
\begin{itemize}

\item[$(\ast)$]for any pair of vertices $v, w\in L$, the edge $\{v, w\}\in L$ given that $(\text{st}_K(v)\cap \text{st}_K(w))\setminus L\neq \emptyset$.

\end{itemize}

Next, we'll obtain a weaker condition (Lemma~\ref{sc_homo_t}) that implies $\text{SC}_K(L)\simeq L$. This is a key tool in the following context to understand the homomoty types of Vietoris-Rips complexes.

Next example shows that SC$_K(L)\simeq L$ even when sometimes they fail the condition $(\ast)$.

\begin{exam} Let $K$ be the complex with maximal faces  $$M(K)=\{\{v_1, v_2\}, \{v_1, v_3\}, \{v_2, v_4, v_5\}, \{v_3, v_4, v_5\}\}$$ and $L$ be the subcomlex with $$M(L)=\{\{v_1, v_2\}, \{v_1, v_3\}, \{v_2, v_4\}, \{v_3, v_4\}\}.$$ Clearly both $K$ and $L$ are clique; and $L\simeq K\simeq S^1$.

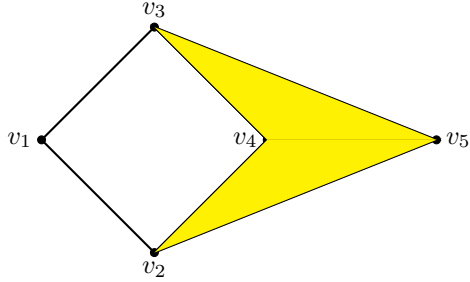
\begin{center}
\begin{tikzpicture}[scale=3]

\draw[fill=yellow!10] (0.5,0) -- (1,0.5) -- (1.75, 0.5) -- cycle;

\draw[fill=yellow!10] (1,0.5) --(0.5,1) -- (1.75, 0.5) -- cycle;
\draw[thick] (0.5,0) -- (0,0.5) -- (0.5,1) -- (1,0.5)  -- cycle;
\draw[thick] (0.5,0) -- (1,0.5) -- (1.75, 0.5)-- cycle;

\draw[thick] (1,0.5) --(0.5,1) -- (1.75, 0.5) --cycle;
\filldraw[black] (0,0.5) circle (0.5pt) node[left]{$v_1$};
\filldraw[black] (0.5,0) circle (0.5pt) node[below]{$v_2$};
\filldraw[black] (1,0.5) circle (0.5pt) node[left]{$v_4$};
\filldraw[black] (0.5,1) circle (0.5pt) node[above]{$v_3$};
\filldraw[black] (1.75,0.5) circle (0.5pt) node[right]{$v_5$};

\fill[yellow] (0.5,0) -- (1,0.5) -- (1.75, 0.5) -- cycle;

\fill[yellow] (1,0.5) --(0.5,1) -- (1.75, 0.5) -- cycle;

\end{tikzpicture}
\end{center}

 Notice that st$_K(v_2)\cap \text{st}_K(v_3)\setminus L\neq \emptyset$, but $\{v_2, v_3\}\notin L$. So this pair of complexes $K$ and $L$ doesn't satisfy condition $(\ast)$.

 \end{exam}




Fix a simplex $s$ in a simplicial complex $K$. The star of $s$ in $K$ is the collection of simplices in $K$ whose union with $s$ is in $K$, i.e. st$_K(s)=\{\tau: \tau\cup s \in K\}$. And the link of $s$ in $K$, lk$_K(s)$, is the collection of simplices which are disjoint from $s$ and whose union with $s$ are in $K$, i.e.,  lk$_K(s)=\{\tau: \tau\cap s=\emptyset\text{ and }\tau\cup s\in K\}$. Clearly  st$_K(s)$ is a cone, hence contractible. Let $L$ be subcomplex of $K$ obtained through removing $s$ and all the simplices containing $s$ as a face. Then clearly $K=L\cup \text{st}_K(s)$; and $L\cap \text{st}_K(s)=\{\tau\in \text{st}_K(s): \tau\cap s \text{ is a proper face of }s\}=\text{lk}_K(s)\ast \{\sigma: \sigma \text{ is a proper face of }s\}$. Also if the inclusion map from $L\cap \text{st}_K(s)$ to $L$ is null-homotopic, then $K$ is homotopy equaivalent to $L\vee \Sigma(L\cap \text{st}_K(s))$ by Lemma~\ref{cup_simp}.

\begin{lemma}\label{sc_homo_t} Let $L$ be a full  subcomplex of a clique complex $K$. Suppose that for any simplex  $s\in \text{SC}_K(L)\setminus L$, lk$_K(s)\cap L$ is contractible.

Fix a set of vertices $\{v_1, v_2, \ldots, v_\ell\}\subseteq L$ and define $L'=L\cup (\bigcup_{i=1}^\ell\text{st}_K(v_i))$. Then the complex $L'$ is homotopy equivalent $L$.

In particular, the star cluster $\text{SC}_K(L)$ is homotopy equivalent to $L$.


\end{lemma}

\begin{proof} Note that $L$ is a full-subcomplex of $K$. So if $(L')^{(0)}\setminus L=\emptyset$, then $L'=L$.

We list all the simplices in $L'\setminus L$ as $\{s_1, s_2, \ldots, s_k\}$ such that $\dim(s_a)\leq \dim(s_{a'})$ if $1\leq a\leq a'\leq k$.  Hence for each $s_a$ in $\{s_1, s_2, \ldots, s_k\}$, the proper faces of $s_a$ appear before $s_a$ in the list; hence $\{s_1, s_2, \ldots, s_a\}$ is also a simplicial complex, denoted by $H_a$. Then, we define $L_a$ to be the complex containing $L$ as a subcomplex and all the simplices in $K$ whose intersections with $H_a$ are also in $H_a$, i.e.,  $$L_a=\{\sigma\in K: \sigma \in L \text{ or } \sigma\cap (H_a)^{(0)}\in H_a \}.$$

Then, it is straightforward to verify that $L' = L_k$. Next, we'll show that for each $a\leq k$, $L_a\simeq L$ by induction. Notice that $s_1$ is a $0$-simplex, i.e. a vertex,  in $\text{SC}_K(L)\setminus L$. So by Lemma~\ref{complex_add_1v},  $L_1\simeq L_1\setminus s_1\vee \text{lk}_{L_1}(s_1)$. Note that $L_1\setminus s_1=L$, $\text{lk}_{L_1}(s_1) = \text{lk}_{K}(s_1)\cap L$ which is contractible by the assumption. Therefore, $L_1\simeq L$.


Now for the purpose of induction, we assume that $L_{a-1}\simeq L$. By the construction, the simplices in $L_a$ not in $L_{a-1}$  are the ones in $K$ whose intersection with $H_{a}^{(0)}$ are exactly $s_a$, i.e. $L_{a-1}$ is a subcomplex of $L_a$ obtained by removing all the simplices containing $s_a$. Hence $L_a=L_{a-1}\cup \text{st}_{L_a}(s_a)$.  We'll show that $L_{a-1}\cap \text{st}_{L_a}(s_a)$ is contractible which,  by the induction assumption and Lemma~\ref{cup_simp}, implies that

$$L_a\simeq L_{a-1}\vee \Sigma(L_{a-1}\cap \text{st}_{L_a}(s_a)) \simeq L_{a-1}\simeq L.$$

Notice that all the proper faces of $\sigma$ appear before $s_a$ in the list; hence as mentioned above,  $L_{a-1}\cap \text{st}_{L_a}(s_a)=\text{lk}_{L_a}(s_a)\ast \{\sigma: \sigma \text{ is a proper face of }s_a\}$. Since $\dim(s_{a'})\leq \dim(s_a)$, $s_{a'}\notin \text{st}_{L_a}(s_a)$, hence lk$_{L_a}(s_a)\subseteq L$; therefore, lk$_{L_a}(s_a)=L\cap \text{lk}_K(s_a)$ which is contractible by the assumption. Therefore $L_{a-1}\cap \text{st}_{L_a}(s_a)$ is a join of a contractible complex with another complex, hence it is contractible.

Inductively $L'=L_k\simeq L_{k-1}\simeq \cdots L_1\simeq L$. \end{proof}


Hence, the following corollary is a direct implication of Lemma~\ref{sc_homo_t}.

\begin{cor}\label{sc_homo_cor} Let $L$ be a full subcomplex of a clique complex $K$. Suppose that lk$_K(s)\cap L$ is either contractible or empty for any simplex $s$ with $s\cap L^{(0)}=\emptyset$. Then the complex SC$_K(L)$ is homotopy equivalent to $L$. \end{cor}

 There are different techniques to describe the homotopy type of $K$. One of them is to start from the a subcomplex $L$ of $K$ and list the maximal faces $\{\sigma_1, \sigma_2, \ldots, \sigma_a\}$ in $K$ which misses $L$. Then it is straightforward to verify as in \cite{FN23} that $K=\text{SC}_K(L)\cup (\bigcup \{\sigma_i: i\leq a\})$. So we could apply Lemma~\ref{cup_simp} to inductively describe $\text{SC}_K(L)\cup (\bigcup \{\sigma_i: i\leq \ell\})$ for $\ell=1, 2, \ldots, a$. Next lemma helps us to understand the homotopy types of some subcomplexes of $K_\sigma$ where $\sigma$ is a simplex.

\begin{lemma}\label{quo_parial_boundary} Let $\sigma$ be a simplex $\{v_1, v_2, \ldots, v_{k+1}, w_1, w_2, \ldots, w_\ell\}$ and $L$ be a complex with the maximal facets $M(L)=\{\sigma\setminus \{v_i\}: i=1, 2, \ldots, k+1\}\cup \{\{v_1, v_2, \ldots, v_{k+1}\}\}$. Then the complex $L$ is homotopy equivalent to $S^k$.  \end{lemma}

\begin{proof} For each $a=1, 2, \ldots, \ell$, define $L_a$ be to be the induced subcommplex of $L$ with the vertex set $\{v_1, v_2, \ldots, v_{k+1}\}\cup \{w_1, w_2, \ldots, w_a\}$. For convenience define $L_0$ to be the induced subcomplex from $L$ with the vertex set $\{v_1, v_2, \ldots, v_{k+1}\}$. Clearly $L=L_\ell$. Note that lk$_{L_1}(w_1)$ is the boundary of the $k$-simplex $\{v_1, v_2, \ldots, v_{k+1}\}$, hence homotopy equivalent to $S^{k-1}$. Therefore,  by Lemma \ref{complex_add_1v}, $$L_1\simeq L_0\vee\Sigma(\text{lk}_{L_1}(w_1))\simeq S^k .$$
For $a\geq 2$, lk$_{L_{a}}(w_a)$ is a cone, therefore contractible; hence $S^k\simeq L_1\simeq L_2\simeq \ldots \simeq L_{\ell}=L $.
\end{proof}
Using a similar approach, we can prove the following lemma.
\begin{lemma}\label{contr_subcompl}   Let $\sigma$ be a simplex $\{v_1, v_2, \ldots, v_{k+1}, w_1, w_2, \ldots, w_\ell\}$.  Fix $p$ such that $1\leq p
\leq \ell$.

Let $L$ be a complex with the maximal facets $M(L)=\{\sigma\setminus \{v_i\}: i=1, 2, \ldots, k+1\}\cup \{\{v_1, v_2, \ldots, v_{k+1}, w_1, w_2, \ldots, w_p\}\}$. Then the complex $L$ is contractible. \end{lemma}


Next, we'll give a general procedure such that we could implement and  carry out the induction on describing the homology types $\text{SC}_K(L)\cup (\bigcup \{\sigma_i: i\leq \ell\})$ for $\ell=1, 2, \ldots, a$ mentioned above and then determine the homotopy type of $K=\text{SC}_K(L)\cup (\bigcup \{\sigma_i: i\leq a\})$.

Recall that for each simplex $\sigma$, $K_\sigma$ denotes the complex generated by $\sigma$, i.e., $K_\sigma$ contains all the nonempty subsets of $\sigma$.

\begin{lemma}\label{sc_cup_maxfacets} Let $L$ be a subcomplex of the complex $K$ with SC$_K(L)\simeq L$. Let $\{\sigma_1, \ldots, \sigma_a, \tau_1, \ldots, \tau_b\}$ be the collection of all maximal facets in $K$ which has empty intersection with $L$ .

Suppose that there is a non-negative integer $d$ satisfying that: \begin{itemize}
\item[i)] $K_{\sigma_i}\cap \text{SC}_K(L)$ is contractible for $i=1, 2, \ldots, a$;

\item[ii)] $K_{\tau_i}\cap \text{SC}_K(L)$ is homotopy equivalent to $S^d$ and is null-homotopic in  $\text{SC}_K(L)$ for $i=1, 2, \ldots, b$;

\item[iii)] $K_{\sigma_i}\cap K_{\sigma_j}$, $K_{\tau_i}\cap K_{\tau_j}$, $K_{\sigma_k}\cap K_{\tau_\ell}$ are all contained in SC$_K(L)$ for $i\neq j$, $i, k\in [a]$, and $j,\ell \in [b].$\end{itemize}

Then the complex $K$ is homotopy equivalent to $$L\vee \bigvee_b S^{d+1}.$$

\end{lemma}

\begin{proof} 
For $\ell=1, 2, \ldots, a$, we define $L_\ell=\text{SC}_K(L)\cup (\bigcup_{i\leq \ell}K_{\sigma_i})$; and for $\ell=a+1, a+2, \ldots, a+b$, $L_\ell=L_a\cup (\bigcup_{i\leq \ell}K_{\tau_{\ell-a}})$. It is straightforward to verify that $K=L_{a+b}$.

First we show that $L_\ell$ is homotopy equivalent to $L$ for each $\ell\leq a$. Since $\text{SC}_K(L)\cap \sigma_1$ is contractible, by Lemma~\ref{cup_simp} and item i) in the assumption, $$L_1\simeq\text{SC}_K(L)\vee \Sigma(\text{SC}_K(L)\cap \sigma_1)\simeq \text{SC}_K(L)\simeq L.$$

Assume the result holds for all number $\ell-1$ with $2\leq \ell\leq a$. Note that by item iii) in the assumption, $L_{\ell-1}\cap K_{\sigma_\ell}= \text{SC}_K(L)\cap K_{\sigma_\ell}$. Hence by Lemma~\ref{cup_simp} and item i) in the assumption again, $L_\ell\simeq L$.

Next, we show that for $\ell=a+1, a+2, \ldots, a+b$, $L_\ell$ is homotopy equivalent to $L\vee (\bigvee_{\ell-a}S^{d+1})$. By item iii) in the assumption, $L_{a}\cap K_{\tau_{1}}=SC_{K}(L)\cap K_{\tau_1}$. By item ii) and Lemma~\ref{null_homo}, $ L_{a}\cap K_{\tau_{1}}$ is null-homotopic in $L_a$. By item ii) in the assumption $L_{a}\cap K_{\tau_{1}}$ is homotopic equivalent to $S^d$. So by Lemma~\ref{cup_simp}, $$L_{a+1}\simeq L_a\vee \Sigma(L_{a}\cap K_{\tau_{1}})\simeq L\vee S^{d+1}.$$

The proof for the general $\ell\geq a+1$ is similar and we skip it here. This shows that $K=L_{a+b}\simeq  L\vee (\bigvee_b(S^{d+1}))$.
\end{proof}
Similarly, we get the following corollary.
\begin{cor}\label{sc_cup_maxfacets_c} Let $L$ be a subcomplex of the complex $K$. Let $\{\sigma_1, \ldots, \sigma_a\}$ be the collection of all maximal facets in $K$ which are not in $L$.  For $i=1, 2, \ldots, a$, define $K_{\sigma_i}$ to be the complex generated by $\sigma_i$.

Suppose that  that: \begin{itemize}
\item[i)] $K_{\sigma_i}\cap L$ is contractible for $i=1, 2, \ldots, a$;

\item[ii)] $\sigma_i\cap \sigma_j$ are all in $L$ for $i\neq j$. \end{itemize}

Then the complex $K$ is homotopy equivalent to $L$.

\end{cor}
The next lemma allows us to characterize  the maximal faces in a subcomplex $L$ of $K$ if the collection of maximal faces in $K$, $M(K)$, is known.
\begin{lemma}\label{mx_f_subcmplx} Let $L$ be a full subcomplex of $K$ with the collection of maximal faces $M(K)$. For any maximal face $\sigma_L\in L$, there is a simplex $\sigma_K\in M(K)$ such that $\sigma_L=\sigma_K\cap L$. \end{lemma}
\begin{proof} Fix a maximal facet  $\sigma_L\in L$. Then there exists a simplex $\sigma_{K}\in M(K)$ such that $\sigma_L\subseteq \sigma_K$. Take a vertex $v\in \sigma_K\cap L$. Because $L$ is a full subcomplex, $\sigma_L\cup \{v\}$ is also a simplex. Since $\sigma_L$ is maximal in $L$, $v\in \sigma_L$. This finishes the proof.
\end{proof}

\section{Maximal facets and Homotopy types of $\mathcal{VR}(\mathcal{F}_1^m\cup \mathcal{F}_2^m\cup \mathcal{F}_3^m; 3)$} \label{mx_f1f2f3_sec}

Starting from this section, each vertex of a complex is a subset of $[m]$ and we'll use $A$, $B$, $C$, or $D$ to represent them. The following lemma gives a characterization of two vertices $A$ and $B$ with $d(A, B)\leq 3$ which is easy to verify.

\begin{lemma}\label{dist_3} For any $A, B\in Q_m$ with $|A|\leq |B|-2$, $d(A, B)\leq 3$ if and only if $A\subseteq B$.

For any $A, B\in Q_m$ with $|A|= |B|-1$, $d(A, B)\leq 3$ if and only if $|A\setminus B|\leq 1$.
\end{lemma}

Let $A$  be a subset of  $[m]$ and $\mathcal{F}$ be a subcollection of $Q_m$. Define $N_\mathcal{F}(A)=\{B\in \mathcal{F}: |B\Delta A|=1\}$ and  $N_\mathcal{F}[A] =\{A\}\cup N_\mathcal{F}(A)$. For $i_1, i_2, \ldots, i_k\in [m]$, $A^{i_1, i_2, \ldots, i_k}=A\Delta \{i_1, i_2, \ldots, i_k\}$. Then $H_A^{i_1, i_2, i_3} =\{A, A^{i_1, i_2}, A^{i_1, i_3}, A^{i_2, i_3}\}$.

Let $K=\mathcal{VR}(Q_m; 3)$ and $M(K)$ be the collection of maximal faces in $K$. Shukla in \cite{Shu22} proved that $M(K)$ could be divided into the following three collections:

\begin{itemize}
\item[] $\mathcal{A}=\{N[A]\cup H_A^{i_1, i_2, i_3}: A\subseteq [m] \text{ and }i_1, i_2, i_3\in [m]\}$;

\item[] $\mathcal{B}=\{N[A]\cup N[B]: A, B\subseteq [m] \text{ and }|A\Delta B|=1\}$;

\item[] $\mathcal{C}=M(K)\setminus (\mathcal{A}\cup \mathcal{B})$ which is a collection of $7$-simplices.
\end{itemize}

Also, the simplices in $\mathcal{C}$ is in the form of the maximal simplices in the join of $8$-copies of $S^0$:
\begin{multline*}
    \overline{\{A, A^{i_1, i_2, i_3,i_4}\}}\ast \overline{\{A^{i_1}, A^{i_2, i_3,i_4}\}}\ast \overline{\{A^{i_2}, A^{i_1, i_3,i_4}\}}\ast \overline{\{A^{i_3}, A^{i_1, i_2,i_4}\}}\ast \overline{\{A^{i_4}, A^{i_1, i_2,i_3}\}}\ast \\ \overline{\{A^{i_1, i_2}, A^{ i_3,i_4}\}}\ast \overline{\{A^{i_1, i_3}, A^{i_2, i_4}\}}\ast \overline{\{A^{i_1, i_4}, A^{i_2, i_3}\}}
\end{multline*}
which are not in $\mathcal{A}$ or $\mathcal{B}$. Here, $A\subset [m]$ and $i_1, i_2, i_3, i_4\in [m]$; also $\overline{\{A, B\}}$ represents the full subcomplex of $K$ with vertices $A$ and $B$.

In the complex $\mathcal{VR}(\mathcal{F}^m_n; 3)=\mathcal{VR}(\mathcal{F}^m_n; 2)$, it is proved in \cite{FN23} that the maximal simplices are in the forms $N_{\mathcal{F}^m_n}(A)$ or $N_{\mathcal{F}^m_n}(B)$ with $|A|=n+1$ and $|B|=n-1$. To determine the collection of maximal faces in $\mathcal{VR}(\mathcal{F}_n^m \cup \mathcal{F}_{n+1}^m; 3)$, by Lemma~\ref{mx_f_subcmplx}, we need to look at all the possible intersections  of the maximal faces in $\mathcal{VR}(Q_m; 3)$ with $\mathcal{F}_n^m \cup \mathcal{F}_{n+1}^m$. By going through all these intersections, we get the following result.

\begin{lemma} The maximal faces in the complex $K=\mathcal{VR}(\mathcal{F}_n^m \cup \mathcal{F}_{n+1}^m; 3)$ are in the following forms:

\begin{itemize}

\item[a)] $N_K [A]\cup N_K[B]$ for $A\subset B$ and i), $|A|=n-1, |B|=n$; ii) $|A|=n, |B|=n+1$; or iii) $|A|=n+1, |B|=n+2$.

\item[b)] $N_K [A]\cup H_A^{i_1, i_2, i_3}$ for i) $|A|=n-1$ and $i_1, i_2, i_3\notin A$;
or,  ii) $|A|=n+2$ and $i_1, i_2, i_3\in A$. 
\end{itemize}
\end{lemma}






It is easy to see that $\mathcal{F}_1^n$ forms a simplex in $\mathcal{VR}(\mathcal{F}_1^n \cup \mathcal{F}_{2}^n; 3)$ since the distance of each pair of different vertices in $\mathcal{F}_1^n$ is equal to $ 2$. Hence it is straightforward to verify that there are two types of maximal faces in the complex $\mathcal{VR}(\mathcal{F}_1^n \cup \mathcal{F}_{2}^n; 3)$ listed in the corollary below.
\begin{cor}\label{max_face_1_2} The maximal faces in the complex $K=\mathcal{VR}(\mathcal{F}_1^n \cup \mathcal{F}_{2}^n; 3)$ are  the intersection of $K^{(0)}$ with the following simplices for $i_1, i_2, i_3\in [n]$:

\begin{itemize}
\item[i)] $N[\emptyset]\cup N[i_1 ]$;


\item[ii)] $N[\emptyset]\cup H_{A}^{i_1, i_2, i_3}$ where $A=\emptyset$.

\end{itemize}

\end{cor}



Next, we'll investigate the homotopy type of $\mathcal{VR}(\mathcal{F}_1^n\cup \mathcal{F}_2^n\cup \mathcal{F}_3^n; 3)$ which is used in the proof of the key lemma, Lemma~\ref{lk_A_preceq}. For $n= 3$, $\mathcal{VR}(\mathcal{F}_1^n\cup \mathcal{F}_2^n\cup \mathcal{F}_3^n; 3)$ is a cone, hence contractible. For $n\geq 4$, we'll use Lemma~\ref{complex_add_1v} determine the homotopy type of the complex $K=\mathcal{VR}((\mathcal{F}_1^n\cup \mathcal{F}_{2}^n\cup\mathcal{F}_{3}^n)\cap \mathcal{F}_{\preceq A}^n; 3)$. So it is essential to determine the homotopy type of lk$_K(A)$.

\begin{lemma}\label{lk_A_f1f2f3} Fix $n\geq 4$ and $A=i_1i_2i_3$ with $i_1\geq 2$. Consider the simplicial complex  $K=\mathcal{VR}((\mathcal{F}_1^n\cup \mathcal{F}_{2}^n\cup\mathcal{F}_{3}^n)\cap \mathcal{F}_{\preceq A}^n; 3)$. Then,

$$\text{lk}_K(A)\simeq \bigvee_{ i_1-1 } S^5.$$
\end{lemma}

\begin{proof}
Let $L=\text{lk}_K(A)$. For $\ell=0,1, \ldots, n$, define $\mathcal{G}_\ell$  in the following way:
\begin{itemize}
\item[1)] $\mathcal{G}_0=\{i_1, i_2, i_3, i_1i_2, i_1i_3, i_2i_3\}$;
\item[2)] $\mathcal{G}_\ell=\emptyset$ if $\ell = i_1, i_2,$ or $i_3$;
\item[3)] $\mathcal{G}_\ell$ is consisting of all of the vertices in $L$ containing $\ell$ if $\ell \neq 0, i_1, i_2,$ or $i_3$
\end{itemize}
It is straightforward to see that $L^{(0)}=\bigcup_{\ell=0}^n \mathcal{G}_\ell$ and $\mathcal{G}_\ell$ forms a simplex in $L$ for each $\ell\neq  i_1, i_2$ or $i_3$. For $\ell<i_1$, $\mathcal{G}_\ell=\{\ell i_1, \ell i_2, \ell i_3, \ell i_1 i_2, \ell i_1 i_3, \ell i_2 i_3\}$ forms a $5$-simplex in $L$.  For $i_1<\ell<i_2$, $\mathcal{G}_\ell=\{i_1\ell, \ell i_2, \ell i_3, i_1\ell i_2, i_1\ell i_3\}$ forms a $4$-simplex since $\ell i_2 i_3\succ A$, i.e. $\ell i_2 i_3\notin K$, hence not in $L$. For similar reason,  when $i_2<\ell<i_3$,  $\mathcal{G}_\ell=\{i_1\ell, i_2\ell, \ell i_3, i_1i_2\ell\}$ forms a $3$-simplex in $L$; and when $\ell>i_3$,
 $\mathcal{G}_\ell=\{i_1\ell, i_2\ell, i_3\ell\}$ forms a $2$-simplex in $L$ .

\medskip

We define $L_\ell=\mathcal{VR}(\bigcup_{i=0}^\ell\mathcal{G}_i; 3)$ for $\ell=0, 1, \ldots, n$. Clearly,  $L=L_n$. We claim that SC$_{L_\ell}(L_{\ell-1})$ is homotopy equivalent to $L_{\ell-1}$ for each $\ell=1, 2, \ldots, n$. Clearly $\mathcal{G}_\ell \subset \text{SC}_{L_\ell}(L_{\ell-1})$ and we'll show that lk$_{L_\ell}(s)\cap L_{\ell-1}$ is either empty or contractible for each simplex  $s$ in  $\mathcal{VR}(\mathcal{G}_\ell; 3)$. Then the claim holds by Corollary~\ref{sc_homo_cor}. Take a simplex $s$ in $\mathcal{VR}(\mathcal{G}_\ell; 3)$. Without loss of generality, we assume $\ell <i_1$.  First, we assume that $s$ is a $0$-simplex in $\mathcal{VR}(\mathcal{G}_\ell; 3)$ and again without loss of generality, assume  that  $s=\{\ell i_1 i_2\}$ or $\{\ell i_1\}$.
\begin{itemize}

\item[i)]
Suppose $s=\{\ell i_1 i_2\}$. Then (lk$_{L_\ell}(s)\cap L_{\ell-1})^{(0)}= \{i_1, i_2, i_1i_2,i_1i_3, i_2i_3\}\cup \{ki_1i_2: k<\ell\}$. Note that the distance between the vertex $i_1$ (or $i_1i_2$) and any other vertex in lk$_{L_\ell}(B)\cap L_{\ell-1}$ is $\leq 3$. Hence, lk$_{L_\ell}(B)\cap L_{\ell-1}$ is a cone, therefore contractible.

\item[ii)] Suppose  $s=\{\ell i_1\}$. Then, (lk$_{L_\ell}(s)\cap L_{\ell-1})^{(0)}=\{i_1, i_2, i_3, i_1i_2,i_1i_3\}\cup\{ki_1i_2: k<\ell \}\cup \{ki_1i_3: k<\ell\}$. Again, the distance between the vertex and any other vertex in lk$_{L_\ell}(B)\cap L_{\ell-1}$ is $\leq 3$.   Therefore, lk$_{L_\ell}(B)\cap L_{\ell-1}$ is a cone, therefore contractible.

\end{itemize}
If $s=\mathcal{G}_\ell$, then lk$_{L_\ell}(s)\cap L_{\ell-1}=\emptyset$. Similarly, it is straightforward to show that lk$_{L_\ell}(s)\cap L_{\ell-1}$ is contractible  if $s$ is a $p$-simplex in $\mathcal{VR}(\mathcal{G}_\ell; 3)$ for $p=1, 2, 3, 4$.

\medskip

Next we investigate the homotopy type of $L=L_n$. And we start from the homotoy type of $L_1$.   Let $\sigma_1=\mathcal{G}_1$ which is a $5$-simplex in $K$ and $K_{\sigma_1}$ be the complex generated by $\sigma_1$.  Since $L_0$ is also a complex generated by a $5$-simplex, $L_1= \text{SC}_{L_1}(L_0)\cup K_{\sigma_1}=\Sigma(\text{SC}_{L_1}(L_0)\cap K_{\sigma_1})$ by Lemma~\ref{cup_simp}. Note that st$_{L_1}(i_1i_2)\cap \sigma_1=\{1i_1, 1i_2, 1i_1i_2, 1i_1i_3, 1i_2i_3\}=\sigma_1\setminus\{1i_3\}$  and st$_{L_1}(i_1)\cap \sigma_1=\{1i_1, 1i_2, 1i_3 1i_1i_2, 1i_1i_3\}=\sigma_1\setminus\{1i_2i_3\}$. So, by going through st$_{L_1}(B)$ for each $B\in L_0$, we recognize that $\sigma_1 \notin \text{SC}_{L_1}(L_0)$ but each proper face of  $\sigma_1$ is $\text{SC}_{L_1}(L_0)$, i.e.,  the boundary of $\sigma_1$ is $\text{SC}_{L_1}(L_0)$.  Hence, $\text{SC}_{L_1}(L_0)\cap \sigma_1$ is is the boundary of $\sigma_1$ which is
homotopy equivalent to $S^4$; and hence, $L_1\simeq S^5$.

Now we fix an general $\ell<i_1$ and $\sigma_\ell$ be the simplex with vertices in $\mathcal{G}_\ell$ and assume that $L_{\ell-1}\simeq \text{SC}_{L{\ell}(L_{\ell-1})}\simeq \bigvee_{\ell-1}S^5$.  Let $K_{\sigma_\ell}$ be the complex generated by $\sigma_\ell$. Note that for each $k< \ell$, st$_{L_\ell}(k i_1)\cap K_{\sigma_\ell}\subseteq \text{st}_{K_\ell}(i_1)\cap K_{\sigma_\ell}$ and st$_{L_\ell}(k i_1i_2)\cap K_{\sigma_\ell} \subseteq \text{st}_{K_\ell}(i_1i_2) \cap K_{\sigma_\ell}$. Hence in general,  SC$_{L_\ell}(L_{\ell-1})\cap K_{\sigma_\ell}=\text{SC}_{L_\ell}(L_{0})\cap K_{\sigma_\ell}$. Then similar as above, $\text{SC}_{L_\ell}(L_{0})\cap K_{\sigma_\ell}$ is homotopy equivalent to $S^4$ which is  null-homotopic in both $\text{SC}_{L_\ell}(L_{0})$ and $K_{\sigma_\ell}$. By Lemma~\ref{cup_simp},   $L_\ell\simeq L_{\ell-1} \vee S^5=\bigvee_{\ell}S^5$.  Specifically,  we get that $L_{i_1-1}\simeq \bigvee_{i_1-1} S^5$.


Finally we fix $\ell>i_1$. Then any vertex in $\mathcal{G}_\ell$ with size $3$ must contains $i_1$ because the vertex $\ell i_2i_3$ is not in $K$. Note that when $\ell>i_3$, $\mathcal{G}_\ell$ contains only vertices of size $2$. Then in either cases, the simplex $\sigma_\ell=   \mathcal{G}_\ell$ is contained in $ \text{st}_{K_\ell}(i_1)$, hence SC$_{L_\ell}(L_{\ell-1})=L_\ell$.  So inductively, we get that $L=L_n\simeq L_{i_1-1}\simeq \bigvee_{i_1-1}S^5$.   \end{proof}

Then inductively, we obtain the homotopy type of the complex $\mathcal{VR}(\mathcal{F}_1^n\cup \mathcal{F}_2^n\cup \mathcal{F}_3^n; 3)$ in the following theorem.

\begin{theorem}\label{homo_f1uf2uf3} For $n\geq 4$, the Vietoris-Rips complex $\mathcal{VR}(\mathcal{F}_1^n\cup \mathcal{F}_2^n\cup \mathcal{F}_3^n; 3)$ is homotopy equivalent to ${n\choose 4}$-many copies of $S^6$.
\end{theorem}

\begin{proof} First, let $K=\mathcal{VR}(\mathcal{F}_1^n\cup \mathcal{F}_2^n\cup \mathcal{F}_3^n; 3)$ and we denote $\mathcal{G}=\mathcal{F}_1^n\cup \mathcal{F}_2^n \cup \{B: B\subset [n], |B|=3, \text{ and }1\in B\}$. The since the distance between $\{1\}$ and any other set in $\mathcal{G}$ is $\leq 3$, the complex $\mathcal{VR}(\mathcal{G}; 3)$  is a cone; hence it is contractible.

For each $A\subset [n]$, we define $L_A= \mathcal{VR}((\mathcal{F}_1^n\cup \mathcal{F}_2^n\cup \mathcal{F}_3^n)\cap \mathcal{F}_{\preceq A}; 3)$. If $A=\{2, 3, 4\}$, lk$_{L_A}(A)$ is homotopy equivalent to $S^5$ by Lemma~\ref{lk_A_f1f2f3}, hence $L_{A}$ is homotopy equivalent to $S^6$ by Lemma~\ref{cup_simp}. For any nonempty set $A\subseteq [n]$, denote $A^-$ be the immediate predecessor of $A$ in the total ordering `$\prec$' in $\mathcal{P}([n])$.  Note that any mapping from a wedge sum of $S^5$'s to a wedge sum of $S^6$'s is null-homotopic.  Inductively, for $A=i_1i_2i_3$ with $i_1\geq 2$, $L_A$ is homotopy equivalent to $L_{A^-}\vee (\bigvee_{i_1-1}S^6)$ by Lemma~\ref{cup_simp} and Lemma~\ref{lk_A_f1f2f3}. Hence inductively, the following holds: $$K\simeq \bigvee_{\sum_{A\in \mathcal{F}^n_3} (\min A-1)} S^6. $$

For each $k=2, 3, \ldots, n-2$, there are ${n-k\choose 2}$ many $3$-subsets of $[n]$ containing $k$ and $k$ being minimal of the sets. Also note that there is no $3$-subset of $[n]$ containing $k$ as minimal for $k=n-1$ or $n$.  So the complex $K$ is homotopy equivalent to the following:

$$\bigvee_{\sum_{k=2}^{n-2}(k-1)\cdot {n-k\choose 2}}S^6.$$

By an induction, it is straightforward to see that $$\sum_{k=2}^{n-2}(k-1)\cdot {n-k\choose 2} ={n\choose 4}.$$ And, this finishes the proof of the theorem. 
\end{proof}




\section{Embedding \lowercase{lk}$_{\mathcal{VR}(\mathcal{F}^m_{\preceq A};3)}(A)$ in the complex $\mathcal{VR}(\mathcal{F}^m_{\prec A};3)$ for $A\subseteq [m]$}\label{lk_A_sec}

Pick $m\geq 5$. Let $A$ be a subset of $[m]$ with size $\geq 4$. We divide this section in two parts. In the first part, we'll determine the homotopy types of the link of $A$ in the complex  $\mathcal{VR}(\mathcal{F}^m_{\preceq A}; 3)$ (Lemma~\ref{lk_A_preceq}), namely lk$_{\mathcal{VR}(\mathcal{F}^m_{\preceq A}; 3)}(A)$ ; in the second part, we'll show that the complex lk$_{\mathcal{VR}(\mathcal{F}^m_{\preceq A};3)}(A)$ is null-homotopic in  $\mathcal{VR}(\mathcal{F}^m_{\prec A}; 3)$ by finding subcomplex $L$ of $\mathcal{VR}(\mathcal{F}^m_{\prec A}; 3)$ which contains  lk$_{\mathcal{VR}(\mathcal{F}^m_{\preceq A};3)}(A)$ but homotopy equivalent to a wedge sum of spheres with dimension $3$ (Lemma~\ref{homo_null}). These two results allow us to inductively describe the homotopy types of $\mathcal{VR}(Q_m; 3)$ in next section.

\subsection{Homotopy types of lk$_{\mathcal{VR}(\mathcal{F}^m_{\preceq A}; 3)}(A)$ for $|A|\geq 4$}

Note that for any two subsets $A$ and $B$ of $[m]$, $d(A, B)=|A\Delta B|=|([m]\setminus A)\cup ([m]\setminus B)|=d([m]\setminus A, [m]\setminus B)$. Therefore, if  $\mathcal{F}^m$ be a collection of subsets of $[m]$, then  the metric space  $\mathcal{F}^m_c$ is isometric to $\mathcal{F}^m$, here  $\mathcal{F}^m_c$ is  be the collection of subsets of $[m]$ whose complement is in $\mathcal{F}^m$. Hence we get the following result.

\begin{lemma}\label{homo_compl} For any $r\geq 0$ and any collection $\mathcal{F}^m$ of subsets of $[m]$, the simplicial complex $\mathcal{VR}(\mathcal{F}^m; r)$ is isomorphic to $\mathcal{VR}(\mathcal{F}^m_c; r)$. \end{lemma}

\begin{lemma}\label{mx_facets} Let $A=i_1i_2\cdots i_n$ be a subset of $[m]$. Fix $\ell \notin A$ and define $\mathcal{G}_\ell=\{B:  B\setminus A=\{\ell\}\text{ and }|B|=n \text{ or }n-1\}$. Let $K=\mathcal{VR}(\mathcal{G}_\ell; 3)$.

Then the maximal facets in $K$ are the intersections of $K$ with one of the following:

\begin{itemize}
\item[i)] $N[\{\ell\}\cup A]\cup  N[\{\ell\}\cup B]$ where $B$ is an $n-1$-subset of $A$.

\item[ii)] $N[\{\ell\}\cup A] \cup H_{\{\ell\}\cup A}^{i_a, i_b, i_c}$ where $i_a, i_b, i_c\in A$.
\end{itemize}

\end{lemma}

\begin{proof} Let $\mathcal{H}$ be the collection of all  subsets of $A$ with sizes $n-1$ or $n-2$ and $\mathcal{H}^c$ the collection of all  subsets of $A$ with sizes $1$ or $2$. Clearly, the metric space $\mathcal{G}_\ell$ is isometric to $\mathcal{H}$. By the operation of taking complement, $\mathcal{H}$ is isometric to $\mathcal{H}^c$, which is clearly isometric to $\mathcal{F}_1^n\cup \mathcal{F}_2^n$. Then by Corollary~\ref{max_face_1_2}, the maximal faces in  the complex $\mathcal{VR}(\mathcal{G}_\ell;3)$ are the ones in the statement of lemma. \end{proof}

\begin{lemma}\label{lk_A_preceq} Let $A=i_1i_2\cdots i_n$ be a subset of $[m]$ with $n\geq 4$. Define $r_k=i_k-(i_{k-1}+1)$ for $k=4, 5, \ldots, n$; and also define $K=\mathcal{VR}(\mathcal{F}^m_{\preceq A}; 3)$.

Then The link of $A$ in the complex $K$, namely lk$_K(A)$, is homotopy equivalent to

$$(\bigvee_{n\choose 4} S^6)\vee \bigvee_{\sum_{k=4}^{n} r_k{k-1\choose 3}}S^3.$$

\end{lemma}

\begin{proof} Define $\mathcal{G}_0$ to be the collection of subsets of $A$ with sizes $n-1$, $n-2$, and $n-3$. By a natural transformation, $\mathcal{VR}(\mathcal{G}_0; 3)$ is isomorphic to $\mathcal{VR}(\mathcal{F}_{n-1}^n \cup \mathcal{F}_{n-2}^n \cup \mathcal{F}_{n-3}^n; 3)$ which is isomorphic to $\mathcal{VR}(\mathcal{F}_{1}^n \cup \mathcal{F}_{2}^n \cup \mathcal{F}_{3}^n; 3)$ by Lemma~\ref{homo_compl}. Hence by Theorem~\ref{homo_f1uf2uf3}, the complex $\mathcal{VR}(\mathcal{G}_0; 3)$ is homotopy equivalent to a wedge sum of ${n\choose 4}$-many $S^6$.

We denote $K_A=\text{lk}_K(A)$. For each $k= 1, 2, \ldots, m$, define $\mathcal{G}_k=\emptyset$ if $k=i_1, i_2, \ldots,$ or $i_n$; otherwise, $\mathcal{G}_k=\{B\in K_{A}^{(0)}: k\in B\}$. By Lemma~\ref{dist_3}, it is straightforward to verify that for any  $B\in \mathcal{G}_k$, $B=\{k\}\cup C$ where $C$ is either an $(n-1)$-subset or an $(n-2)$-subset of $A$; hence,
 clearly $\mathcal{G}_k\cap \mathcal{G}_{k'}=\emptyset$ when $k\neq k'$. Also notice that for any $k>i_p$, the vertices in $\mathcal{G}_k$ with size $n$ must contain all of  $\{i_1, i_2, \ldots, i_p\}$ because replacing any of elements in $\{i_1, i_2, \ldots, i_p\}$ by $k$ yields a vertex $\succ A$ which is not in the complex $K$.

\medskip

For each $k\leq m$, we define $L_k=\mathcal{VR}(\bigcup_{\ell=0}^k\mathcal{G}_\ell;3)$. So $L_0$ is homotopy equivalent to ${n\choose 4}$-many copies of $S^6$. We claim that SC$_{L_{k}}(L_{k-1})$ is homotopy equivalent to $L_{k-1}$ for each $k=1, 2, \ldots, m$. This result trivially holds when $k\in A$, i.e. $k=i_1, i_2,\ldots,$ or $i_n$  since $G_k=\emptyset$ for such $k$. Now fix $k$ such that $k\notin A$.  We'll show that for any simplex $s\in \mathcal{VR}(\mathcal{G}_{k}; 3)$, lk$_{L_{k}}(s)\cap L_{k-1}$ is either contractible or an empty set; then by Corollary~\ref{sc_homo_cor}, the claim holds. First We'll show the results hold for the $0$-simplex in $\mathcal{VR}(\mathcal{G}_{k}; 3)$.

By the discussion above, there are at most two different kinds of $0$-simplices, i.e. vertices,  in $\mathcal{VR}(\mathcal{G}_{k}; 3)$, with either size $n$ or size $n-1$. We divide the discussion into two cases below if the corresponding vertex of the size exists in $\mathcal{G}_k$.

\begin{itemize}

\item[i)] Assume $s=\{B\}$ where $B\in \mathcal{G}_{k}$ with size $n$. Then $B$ is consisting of $k$ and an $(n-1)$-subset of $A$, denoted by $C$, i.e. $B=\{k\}\cup C$. Note that $C$ is also a vertex in $\mathcal{G}_0$. Then by Lemma~\ref{dist_3}, the intersection of the link of $s$ in $L_k$, lk$_{L_k}(s)$, with  $L_{k-1}$ contains the following vertices if possible: 1) vertices of size $n$ consisting of numbers in $C$ and a number $k'<k$ with $k'\neq i_1, i_2, \ldots, i_n$; 2) vertices of size $n-1$ containing an $(n-2)$-subset of $C$; 3) vertices of size $n-2$ which are subsets of $C$. All these vertices have a distance $\leq 3$ with the vertex $C$. Hence   lk$_{L_k}(s)\cap L_{k-1}$ is a cone which is contractible.

\item[ii)] Assume $s=\{B\}$ where  $B\in \mathcal{G}_{k}$ with size $n-1$. Then $B$ can be written as $\{k\}\cup C$ where $C$ is an $(n-2)$ subset of $A$. Again by Lemma~\ref{dist_3}, the vertices in lk$_{L_k}(s)\cap L_{k-1}$ is in one of the following forms: 1) vertices of size $n$ containing $C$ ; 2) vertices of size $n-1$ containing of $C$; 3) vertices $D$ of size $n-2$ with $|D\setminus C|\leq 1$ . All these vertices have a distance $\leq 3$ with the vertex $C$. Hence, again lk$_{L_k}(s)\cap L_{k-1}$ is a cone, hence contractible.

\end{itemize}

For higher dimensional simplices in $\mathcal{VR}(\mathcal{G}_{k}; 3)$, the proofs are similar and we skip them.

\medskip

Next we prove that $L_{\ell}\simeq L_{\ell-1}$ for $\ell\leq i_3$, or $\ell\geq i_n$, or $\ell\in A$. This trivially holds when $\ell\in A$ since $\mathcal{G}_\ell=\emptyset$ for such $\ell$. Pick $\ell>i_n$ if exists. Then $\mathcal{G}_\ell$ contains only vertices of size $(n-1)$ which can be written as $\{\ell\}\cup C$ with $C$ being an $(n-2)$-subset of $A$. By Lemma~\ref{mx_facets}, there are two types of maximal facets in $\mathcal{VR}(\mathcal{G}_\ell; 3)$, namely $N[j_1, j_2, \ldots, j_{n-1}, \ell]\cap \mathcal{G}_\ell$ and $N[j_1, \ldots, j_{n-3}, \ell] \cap \mathcal{G}_\ell$ for some $j_1, j_2, \ldots, j_{n-1}\in A$. It is straightforward to verify that for any $j_1, j_2, \ldots, j_{n-1}\in A$, $(N[j_1, j_2, \ldots, j_{n-1}, \ell]\cap \mathcal{G}_\ell)\in \text{st}_{L_\ell}(j_1j_2\cdots j_{n-1})$ and $N[j_1, \ldots, j_{n-3}, \ell] \cap \mathcal{G}_\ell\in \text{st}_{L_\ell}(j_1j_2\cdots j_{n-2})$. Hence SC$_{L_\ell}(L_{\ell-1})=L_\ell$ for such $\ell$ and hence  $L_\ell =\text{SC}_{L_\ell}(L_{\ell-1})\simeq L_{\ell-1}$.

Now we pick $\ell$ with $i_2<\ell<i_3$ if such an $\ell$ exists. For $\ell$ with $\ell<i_1$ or $i_1<\ell<i_2$, the proofs are very similar and we'll skip them. Then by the discussion above, the vertex  in $\mathcal{G}_\ell$ can be written as: either 1), $\{\ell\}\cup C$ with $C$ being an $(n-2)$ subset of $A$; or 2), $\{\ell\}\cup D$ with $D$ being an $(n-1)$-subset of $A$ such that $i_1, i_2\in D$. 
Define $\mathcal{H}'$ to be the collection of vertices in $\mathcal{G}_\ell$ with size $n$. So each vertex in $\mathcal{H}'$ must contain $i_1, i_2$, and $\ell$, i.e., $$\mathcal{H}'=\{i_1i_2\ell i_3\cdots \hat{i_a} \cdots i_n: a=3, 4, \ldots,\text{ or } n\}. $$  
Note that the distance between any $(n-1)$-subset of $A$ and any vertex in $\mathcal{H}'$ is at most $3$.  By Lemma~\ref{mx_facets}, there are two types of maximal faces in $\mathcal{VR}(\mathcal{G}_\ell; 3)$. We verify the conditions in Corollary~\ref{sc_cup_maxfacets_c} in the following, i.e.  we'll show that the intersection of either kind of the maximal faces with SC$_{L_\ell}(L_{\ell-1})$ is contractible and the intersection of each pair of maximal facets are SC$_{L_\ell}(L_{\ell-1})$. In i) and ii), we show that SC$_{L_\ell}(L_{\ell-1})\cap \sigma$ is contractible for each maximal facets in $\mathcal{VR}(\mathcal{G}_\ell; 3)$. In iii), we show that the intersection of each pair of maximal facets $\sigma$ and $\sigma'$ is in SC$_{L_\ell}(L_{\ell-1})$.  Then by Corollary~\ref{sc_cup_maxfacets_c}, SC$_{L_\ell}(L_{\ell-1})\simeq L_{\ell}$.

\begin{itemize}

\item[i)] Assume that $\sigma =N[\{\ell\}\cup A]\cup  N[\{\ell\}\cup B]\cap \mathcal{G}_\ell$ where $B$ is an $(n-2)$-subset of $A$. Hence $B$ is a vertex in $L_0=\mathcal{VR}(\mathcal{G}_0; 3)$. Then $\sigma \in \text{st}_{L_\ell}(B)$ which  means that SC$_{L_\ell}(L_{\ell-1})\cap \sigma =\sigma$ hence contractible;

\item[ii)] Without loss of generality, we assume that $\sigma =(N[\{\ell\}\cup A]\cup H_{\{\ell\}\cup A}^{i_1, i_2, i_3})\cap \mathcal{G}_\ell$. Then $\sigma=\mathcal{H}'\cup H_{\{\ell\}\cup A}^{i_1, i_2, i_3}$. Let $D=A\setminus \{i_1, i_2, i_3\}$. And then let $B_1=\{\ell, i_1\}\cup D=\{\ell\}\cup A\setminus \{i_2, i_3\}$, $B_2=\{\ell, i_2\}\cup D$, and $B_3=\{\ell, i_3\}\cup D$; hence,  $H_{\{\ell\}\cup A}^{i_1, i_2, i_3}=\{B_1, B_2, B_3\}$. By going through the star of vertices in $L_{\ell-1}$, we get that 1), $\{B_1, B_2, B_3, (\{\ell\}\cup A)\setminus \{i_3\}\}= \text{st}_{L_\ell}(D)\cap \mathcal{G}_\ell$; 2) $\mathcal{H}'\cup \{B_1, B_2\}= \text{st}_{L_\ell}(A\setminus\{i_3\})\cap \mathcal{G}_\ell$; 3) $\mathcal{H}'\cup \{B_1, B_3\}= \text{st}_{L_\ell}(A\setminus\{i_2\})\cap \mathcal{G}_\ell$; 4) $\mathcal{H}'\cup \{B_2, B_3\}= \text{st}_{L_\ell}(A\setminus\{i_1\})\cap \mathcal{G}_\ell$.

For any other vertex $C$ in $L_{\ell-1}$, $\text{st}_{L_\ell}(A\setminus\{i_1\})\cap \mathcal{G}_\ell$ a subset of one of the following: 1), $\text{st}_{L_\ell}(D)\cap \mathcal{G}_\ell$; 2), $\text{st}_{L_\ell}(A\setminus\{i_c\})\cap \mathcal{G}_\ell$; 3), $\text{st}_{L_\ell}(A\setminus\{i_b\})\cap \mathcal{G}_\ell$; or 4), $\text{st}_{L_\ell}(A\setminus\{i_a\})\cap \mathcal{G}_\ell$.  Hence SC$_{L_\ell}(L_{\ell-1}) \cap \sigma$ is a complex with maximal faces: 1), 2), 3), and 4). Then by
 Lemma~\ref{contr_subcompl}, $\text{SC}_{L_\ell}(L_{\ell-1})\cap \sigma$ is contractible.

\item[iii)] Let $\sigma$, $\sigma'$ be maximal facets in $\mathcal{VR}(\mathcal{G}_\ell; 3)$ but not in SC$_{L_\ell}(L_{\ell-1})$. By the discussion in ii) above, it is straightforward to verify that their intersection is SC$_{L_\ell}(L_{\ell-1})$. \end{itemize}

\medskip
Fix $k\geq 4$ with $i_k-(i_{k-1}+1)\neq 0$. Lastly we prove that for $\ell$ with  $i_{k-1}<\ell<i_k$, $L_{\ell}=L_{\ell-1}\vee (\bigvee_{k-1\choose 3}S^3)$. This result establishes that $K_A=L_m\simeq (\bigvee_{n\choose 4} S^6)\vee \bigvee_{\sum_{k=4}^{n} r_k{k-1\choose 3}}S^3$ by an induction. Fix $\ell$ with  $i_{k-1}<\ell<i_k$ and assume that $L_{\ell-1}$ is a wedge sum of spheres with dimensions $6$ and/or $3$. This holds when $\ell=\min\{b: b\in [m]\setminus A\text{ and }b>i_3\}$ because for such an $\ell$, $L_{\ell-1}\simeq\ldots \simeq  L_0$ which is a wedge sum of spheres with dimension $6$. For the simplex $\sigma$ in the form $N[\{\ell\}\cup A]\cup  N[\{\ell\}\cup B]\cap \mathcal{G}_\ell$ or $(N[\{\ell\}\cup A]\cup H_{\{\ell\}\cup A}^{i_a, i_b, i_c})\cap \mathcal{G}_\ell$ with one of $\{i_a, i_b, i_c\}$ $>k$, $\sigma\cap \text{SC}_{L_\ell}(L_{\ell-1})$ is contractible by the discussion in i) and ii) above. Let $\tau$ be a simplex in the form $(N[\{\ell\}\cup A]\cup H_{\{\ell\}\cup A}^{i_a, i_b, i_c})\cap \mathcal{G}_\ell$ with all of $\{i_a, i_b, i_c\}$ $<k$ and $i_a<i_b<i_c$. Note that there are ${k-1\choose 3}$-many simplices in $\mathcal{VR}(\mathcal{G}_\ell; 3)$ in such form. Then we define $\mathcal{H}''$ be the collections of vertices in $\mathcal{G}$ with size $n$. Since $i_a<i_b<i_c<\ell$, each vertex in $\mathcal{H}''$ must contain $\{i_a, i_b, i_c\}$ as a subset. We define $D$, $B_a$, $B_b$, $B_c$ as above, i.e.,  $D=A\setminus \{i_a, i_b, i_c\}$,  $B_a=\{\ell, i_a\}\cup D$, $B_b=\{\ell, i_b\}\cup D$, and $B_c=\{\ell, i_c\}\cup D$. It is straightforward to verify that 1), $\{B_a, B_b, B_c\}= \text{st}_{L_\ell}(D)\cap \mathcal{G}_\ell$; 2) $\mathcal{H}'\cup \{B_a, B_b\}= \text{st}_{L_\ell}(A\setminus\{i_c\})\cap \mathcal{G}_\ell$; 3) $\mathcal{H}'\cup \{B_a, B_c\}= \text{st}_{L_\ell}(A\setminus\{i_b\})\cap \mathcal{G}_\ell$; 4) $\mathcal{H}'\cup \{B_b, B_c\}= \text{st}_{L_\ell}(A\setminus\{i_a\})\cap \mathcal{G}_\ell$. And the intersection of the star of any other vertex in $L_{\ell-1}$ with $\mathcal{G}_\ell$ is a subset of 1), 2), 3) or 4). By Lemma~\ref{quo_parial_boundary}, SC$_{L_\ell}(L_{\ell-1})\cap \sigma$ is homotopy equivalent to $S^2$; so the corresponding inclusion map is null-homotopic.  By going through such simplices, their intersections are also in SC$_{L_\ell}(L_{\ell-1})$.   Therefore by Lemma~\ref{sc_cup_maxfacets}, $L_\ell=L_{\ell-1}\vee (\bigvee_{{k-1\choose 3}}S^3)$. \end{proof}

Then, for any $A=i_1i_2\cdots i_n$ with $n\geq 4$, we define $r_k= i_{k}-i_{k-1}-1$ for $k=4, 5, \ldots, n$ and $$r_A=\sum_{k=4}^{n} r_k{k-1\choose 3}.$$

\subsection{Null-Homotopicness of $\text{lk}_{\mathcal{VR}(\mathcal{F}^m_{\preceq A}; 3)}(A)$ in $\mathcal{VR}(\mathcal{F}^m_{\preceq A}; 3)$  for $|A|\geq 4$}

Let $K=\mathcal{VR}(\mathcal{F}_{\preceq A}; 3)$ with $|A|\geq 4$. Now we know that lk$_K(A)$ is homotopy equivalent to a wedge sum of $S^6$'s and/or $S^3$'s.  In order to eliminate the $S^6$'s in lk$_K(A)$, we need to expand the complex lk$_K(A)$ with more vertices to kill the $6$-dimensional spheres. Note that the $6$-dimensional spheres in lk$_K(A)$ are from the collection of all proper subsets of $A$ with size $n-1$, $n-2$, $n-3$ which is isometric to $\mathcal{F}_1^n\cup \mathcal{F}_2^n\cup \mathcal{F}_3^n$. In Lemmas~\ref{lk_A_null} and~\ref{null_f1f2f3_G} we get a superset of $\mathcal{F}_1^n\cup \mathcal{F}_2^n\cup \mathcal{F}_3^n$ whose Vietoris-Rips complex with scale $3$ is contractible, i.e. the $6$-dimensional spheres are eliminated. The strategy of the proof is same as the proof of Theorem~\ref{homo_f1uf2uf3}.

\begin{lemma}\label{lk_A_null} For $n\geq 4$, let $\mathcal{G}$ be the $4$-subset of $[n]$ which contains $1$. Pick $A=i_1i_2i_3$ with $i_1\geq 2$ and let $K=\mathcal{VR}((\mathcal{F}_{\preceq A}^n\setminus \mathcal{F}_{0}^n) \cup \mathcal{G}; 3)$.  Then, lk$_K(A)$ is contractible. \end{lemma}

\begin{proof} Let $L=\text{lk}_K(A)$ and  divide the vertices in  into collections $\mathcal{H}_1$ and $\mathcal{H}_2$ such that:
\begin{itemize}

\item[1)] $\mathcal{H}_1$ contains the vertices in $L$ with sizes $1, 3, 4$ and all the subsets of $A$ with size $2$;

\item[2)] $\mathcal{H}_2$ contains all the vertices $C$ of $L$ with size $2$ and $|C\cap A|=1$.
\end{itemize}
Let $B=1i_1i_2i_3$ and clearly $B\in \mathcal{H}_1$. It is straightforward to verify that the distance between $B$ and any other vertex in $\mathcal{H}_1$ is $\leq 3$. Hence $L_1=\mathcal{VR}(\mathcal{H}_1; 3)$ is a cone and contractible.

Pick a simplex $s\in\mathcal{VR}(\mathcal{H}_2; 3)$. Without loss of generality assume that $s=\{C\}$ where $C=\{i_1, k\}$ with $k\neq i_1, i_2$, or $i_3$. Then the vertices of  the link of $s$ in $L_1$, $\text{lk}_L(s)\cap L_1$,  is consisting of the following:

\begin{itemize}

\item[1)] all the subsets of $A$ with size $1$;
\item[2)]  the subsets of $A$ with size $2$ and containing $i_1$;
\item[3)] the vertices with size $3$ in $L_1$ containing $i_1$.
\end{itemize}
Hence $\text{lk}_L(s)\cap L_1$ is a cone, i.e. contractible, because the distance between the vertex $i_1$ with any other vertex in $\text{lk}_L(s)\cap L_1$ is $\leq 3$. Similarly, we can prove that $\text{lk}_L(s)\cap L_1$ is contractible for other simplex $s$ in $\mathcal{VR}(\mathcal{H}_2; 3)$. By Lemma~\ref{sc_homo_t}, SC$_L(L_1)$ is homotopy equivalent to $L_1$.

Since $\mathcal{H}_2$ is consisting of vertices with size $2$, $\mathcal{VR}(\mathcal{H}_2; 3)\subset \text{st}_L(D)$ for vertex $D$ in $\mathcal{H}_1$ with size $1$; Hence, SC$_L(L_1)=L$. Therefore, $L\simeq L_1$ is contractible.
\end{proof}

\begin{lemma}\label{null_f1f2f3_G} For $n\geq 4$, let $\mathcal{G}$ be the $4$-subset of $[n]$ which contains $1$. The complex $\mathcal{VR}(\mathcal{F}_1^n\cup \mathcal{F}_2^n\cup \mathcal{F}_3^n\cup \mathcal{G}; 3)$ is contractible. \end{lemma}

\begin{proof} First let $A=\{2, 3, 4\}$ and $\mathcal{H}_0=(\mathcal{F}_{\prec A}^n\setminus \mathcal{F}_{0}^n) \cup \mathcal{G}$. Then the distance between $\{1\}$ and any other vertex in $\mathcal{H}_0$ is $\leq 3$. Hence $L_0 = \mathcal{VR}(\mathcal{H}_0; 3)$ is a cone, hence contractible. Then if $L_A=\mathcal{VR}((\mathcal{F}_{\preceq A}^n\setminus \mathcal{F}_{0}^n) \cup \mathcal{G}; 3)$, $L_A$ is contractible by Lemma~\ref{cup_simp} because both $L_0=L_A\setminus \{A\}$ and lk$_{L_A}(A)$ (by Lemma~\ref{lk_A_null}) are contractible.  Then by an induction, $\mathcal{VR}(\mathcal{F}_1^n\cup \mathcal{F}_2^n\cup \mathcal{F}_3^n\cup \mathcal{G}; 3)$ is contractible. \end{proof}

Now we're ready to prove that there is a subcomplex $L$ of $K=\mathcal{VR}(\mathcal{F}_{\prec A}; 3)$ with $|A|\geq 3$ such that $L$ is homotopy equivalent to a wedge sum of $S^3$'s and lk$_K(A)\subset L$.

\begin{lemma}\label{homo_null} Let $A=i_1i_2\cdots i_n$ be an element in $[m]$ with $n\geq 4$. Let  $K=\mathcal{VR}(\mathcal{F}^m_{\preceq A}; 3)$.

 Then there exists a subcomplex $L$ of $K\setminus \{A\}$ such that:
 \begin{itemize}
 \item[i)] the link of $A$ in the complex $K$, namely lk$_K(A)$, is contained in $L$
 \item[ii)] $L$ is homotopy equivalent to $\bigvee_{r_A}S^3$.

 \end{itemize}\end{lemma}

 \begin{proof} Pick $L$ be the full subcomplex of $K$ with the vertices from lk$_K(A)$ and the vertices which are $(n-4)$-subsets of $A$ which doesn't containing $i_1$.

 The proof of $L\simeq \bigvee_{r_A}S^3$ is almost identical with the proof of Lemma~\ref{lk_A_preceq}. The only difference is that we start from $\mathcal{G}_0$ be the collection of subsets of $A$ of sizes $n-1$, $n-2$, $n-3$, and also the ones with size $n-4$ which doesn't containing $i_1$. Again by a natural transformation, $\mathcal{VR}(\mathcal{G}_0; 3)$ is isomorphic to the complex in Lemma~\ref{null_f1f2f3_G} which is contractible. The rest of the proof is same as the proof of Lemma~\ref{lk_A_preceq}.  \end{proof}

\section{Homotompy types of $\mathcal{VR}(Q_m; 3)$ for $m\geq 5$}\label{main_sec}

In this section, we'll prove the main result in the paper and describe the homotopy types of $\mathcal{VR}(Q_m; 3)$ for $m\geq 5$. We divide the proofs into two steps: firstly, we inductively determine the homotopy types of $\mathcal{VR}(\mathcal{F}^m_{\preceq A}; 3)$ for any vertex $A$ with size $\geq 4$ using Lemma~\ref{complex_add_1v}, Lemma~\ref{lk_A_preceq}, and Lemma~\ref{homo_null}; then, we prove the inductive formulas for the homotopy types of $\mathcal{VR}(Q_m; 3)$ which yields the one in the main theorem.

\begin{theorem}\label{homo_preceqA} Let $A=i_1i_2\cdots i_n$ be an element in $[m]$ with $n\geq 4$. Let  $K=\mathcal{VR}(\mathcal{F}^m_{\preceq A}; 3)$. The the complex $K$ is homotopy equivalent to

$$\mathcal{VR}(\mathcal{F}^m_{\prec A}; 3)\vee (\bigvee_{n\choose 4} S^7) \vee (\bigvee_{r_A} S^4).$$

\end{theorem}

\begin{proof} We'll prove the result by induction on the order `$\prec$' over $Q_m$. First let $A=\{1, 2, 3, 4\}$.  The complex $\mathcal{VR}(\mathcal{F}^m_{\prec A}; 3)$ is a cone because the distance between the empty set and any other vertex is $\leq 3$. Also the link of $A$ in $\mathcal{VR}(\mathcal{F}^m_{\preceq A}; 3)$ is homotopy equivalent to $S^6$ by Lemma~\ref{lk_A_preceq} since $r_A=0$. Then by Lemma~\ref{complex_add_1v}, $\mathcal{VR}(\mathcal{F}^m_{\preceq A}; 3)\simeq S^7$.

Now we take $A=i_1i_2\cdots i_n$ with $n\geq 4$ and assume that $\mathcal{VR}(\mathcal{F}^m_{\prec A}; 3)$ is a wedge sum of $S^7$'s and $S^4$'s, or just $S^7$'s. By Lemma~\ref{homo_null}, there is a subcomplex $L$ of $\mathcal{VR}(\mathcal{F}^m_{\prec A}; 3)$ such that lk$_K(A)\subset L$ and $L\simeq \bigvee_{r_A}S^3$. Note that $L$ is null-homotopic in $\mathcal{VR}(\mathcal{F}^m_{\prec A}; 3)$. So by Lemma~\ref{null_homo}, lk$_K(A)$ is null-homotopic in $\mathcal{VR}(\mathcal{F}^m_{\prec A}; 3)$. Hence by Lemma~\ref{complex_add_1v} and Lemma~\ref{lk_A_preceq},  $$\mathcal{VR}(\mathcal{F}^m_{\preceq A}; 3)\simeq \mathcal{VR}(\mathcal{F}^m_{\prec A}; 3)\vee \Sigma(\text{lk}_K(A))\simeq \mathcal{VR}(\mathcal{F}^m_{\prec A}; 3)\vee (\bigvee_{n\choose 4} S^7) \vee (\bigvee_{r_A} S^4). $$ \end{proof}

\begin{theorem}\label{main} For $m\geq 5$, The complex $\mathcal{VR}(Q_m; 3)$ is homotopy equivalent to $$(\bigvee_{2^{m-4}\cdot{m\choose 4}} S^7) \vee (\bigvee_{\sum_{i=4}^{m-1}2^{i-4}\cdot{i\choose 4}} S^4).$$

\end{theorem}

\begin{proof} By Theorem~\ref{homo_preceqA} and an inductive argument on the total order `$\prec$' over $\mathcal{P}([m])$, the following holds:
\[\mathcal{VR}(Q_m; 3)\simeq (\bigvee_{\sum_{A\subset [m] \text{ and } |A|\geq 4 }{|A|\choose 4}}S^7) \vee (\bigvee_{\sum_{A\subset [m] \text{ and } |A|\geq 4 }r_A} S^4).\]

Note the following equation:  $$\sum_{A\subset [m] \text{ and } |A|\geq 4 }{|A|\choose 4} =\sum_{n=4}^m {m\choose n}\cdot {n\choose 4}=\sum_{n=4}^m {m\choose 4}\cdot {m-4\choose n-4}=2^{m-4}\cdot {m\choose 4}. $$

Hence in the homotopy type of $\mathcal{VR}(Q_m; 3)$ with $m\geq 5$, there are $2^{m-4}\cdot {m\choose 4}$-many folds of $S^7$.

Next we'll count the number of copies of $S^4$ in the homotopy type of $\mathcal{VR}(Q_m; 3)$. When $m=5$, the only $A\subset [m]$ with $r_A\neq 0$ is when $A=\{1, 2, 3, 5\}$ with $r_A=1$. Hence when $m=5$, $$\sum_{A\subset [m] \text{ and } |A|\geq 4 }r_A=\sum_{i=4}^{m-1}2^{i-4}{i\choose 4}=1.$$

Inductively, we assume that the equation holds for $m-1$ with $m>4$. By the definition of $r_A$, the number of folds of $S^4$ is determined by  the gaps of the consecutive pairs which appear in the third position and after. So we'll count many name new such pairs appearing in the vertices of $Q_m$ not in $Q_{m-1}$. All the gaps in $A\in Q_{m-1}$ accounting for the $S^4$ reappearing in   $Q_m$  the vertex $A\cup \{m\}$. Other new pairs accounting for new folds of $S^4$ come in the forms of the pair $\{k, m\}$ with $k\geq 3$.

For each pair $\{k, m\}$ with $3\leq k\leq m-2$, we count the number of subsets of $[m]$ with size $\geq 4$ which contains $\{k, m\}$ but not any number between $k$ and $m$.  There are $\sum_{i=3}^{k} (m-k-1)\cdot {k-1\choose {i-1}} {i\choose 3}$ copies of $S^4$ in the complex $\mathcal{VR}(Q_m; 3)$ contributed by this pair. So overall, we have the following equation:

\begin{align*}
\lefteqn{\sum_{A\subset [m] \text{ and } |A|\geq 4 }r_A-\sum_{A\subset [m-1] \text{ and } |A|\geq 4 }r_A} 
\qquad&
  \\[1ex]
  &=\sum_{i=4}^{m-2}2^{i-4}{i\choose 4}+\sum_{k=3}^{m-2}(\sum_{i=3}^{k} (m-k-1)\cdot {k-1\choose {i-1}} {i\choose 3})
\end{align*}

By an induction on $m$, the right side of the equation above is equal to $2^{m-5}{m-1\choose 4}$ for all $m\geq 6$ and its proof is included in Appendix~\ref{combina} (Proposition~\ref{comb_formular_2}). And this shows that for any $m\geq 5$,

$$\sum_{A\subset [m] \text{ and } |A|\geq 4 }r_A=\sum_{i=4}^{m-1}2^{i-4}{i\choose 4}.$$

This finishes the proof.
\end{proof}

\section{Large Scales $r\geq 4$}\label{conclusion_sec}

In general, one problem about determining the homotopty types of these Vietoris-Rips complexes is to identify the dimensions of spheres in theses homotopy types. Based on the results so far, the dimensions of the spheres in the homotopy types of $\mathcal{VR}(Q_m; r)$ in general start to appear at small $m$ and then only the folds of the spheres increase as $m$ gets larger and larger. For example,  the spheres of dimension $4$ and $7$ in the homotopy types of the complex $\mathcal{VR}(Q_m; 3)$ start to appear at $m=5$; and, the spheres of dimension $3$ in the homotopy types of  $\mathcal{VR}(Q_m; 2)$ start to appear at $m=3$. For any $r\geq 0$, we define $m(r)$ to be the smallest $m^\ast$ such that the homotopy type of $\mathcal{VR}(Q_{m^\ast}; r)$ contains at least one copy of the spheres which appear in the homotopy type of $\mathcal{VR}(Q_{m}; r)$ for arbitrary $m\geq m^\ast$. Hence overall, we have that $m(0)=1$, $m(1)=2, m(2)=3, m(3)=5$.  Hence, the following question is natural:

\begin{ques} What is $m(r)$ for $r\geq 4$? Is $m(4)=7$?\end{ques}

In \cite{AV23}, Adams and Virk identified a class of homology generators in the complex $\mathcal{VR}(Q_m; r)$ of the dimension $2^r-1$ which are all independent, hence found lower bounds on the rank of their  $(2^r-1)$-dimensional homology.     Let $K$ be a simplicial complex and $L$ a subcomplex of $K$. We say that $L$ is cross-polytopal if it is isomorphic to the boundary of a cross-polytope with $2^{d+1}$ vertices; hence such $L$ is homeomorphic to the sphere of dimension $2^d-1$. Going through the proofs in previous sections, we see that each copy  of spheres in the homotopy types of $\mathcal{VR}(Q_m; 3)$ is corresponding to one of its  cross-polytopal subcomplexes; but we don't know whether this is true for all the complexes $\mathcal{VR}(Q_m; r)$.

\begin{ques} Is the complex $\mathcal{VR}(Q_m; r)$ always homotopy equivalent to a wedge sum of spheres? Are all of the spheres in the homotopy types of $\mathcal{VR}(Q_m; r)$ generated by one of its cross-polytopal subcomplexes? \end{ques}

In general the following question is still open. This is an update of the question raised by Adamaszek and Adams in \cite{AA22}.

\begin{ques} What are the homotopy types of $\mathcal{VR}(Q_m; r)$ for $r=4, 5, \ldots, m-2$?

\end{ques}

\medskip

\textbf{Acknowlegments}

\medskip

We would like to thank Henry Adams, Samir Shukla, Anurag Singh, and \v{Z}iga Virk for useful discussions and kind advices.

\newpage
\appendix \section{Couple Identities in Combinatorics}\label{combina}

We denote this section to prove two identities in combinatorics which are used in the proof of the main theorem (Theorem~\ref{main}).

\begin{prop}\label{comb_formular_1} For $m\geq 5$, \[\sum_{i=3}^{m-1}{m-2\choose i-1}{i\choose 3}+\sum_{k=3}^{m-2}(\sum_{i=3}^{k} {k-1\choose {i-1}} {i\choose 3})=2^{m-4}{m-1\choose 3}.\]
\end{prop}
\begin{proof} The result clearly holds for $m=5$. Assume that the result holds for $m-1$ for $m\geq 6$, i.e.

\[\sum_{i=3}^{m-2}{m-3\choose i-1}{i\choose 3}+\sum_{k=3}^{m-3}(\sum_{i=3}^{k} {k-1\choose {i-1}} {i\choose 3})=2^{m-5}{m-2\choose 3}.\]

Then,

\begin{equation*}
\begin{split}
\sum_{i=3}^{m-1}{m-2\choose i-1}{i\choose 3}+\sum_{k=3}^{m-2}(\sum_{i=3}^{k} {k-1\choose {i-1}} {i\choose 3})=2^{m-5}{m-2\choose 3} +\sum_{i=3}^{m-1}{m-2\choose i-1}{i\choose 3}
\end{split}
\end{equation*}

Note that \[2^{m-5}{m-2\choose 3}=\sum_{i=3}^{m-1} {m-2\choose 3}{m-5\choose i-3}= \sum_{i=3}^{m-1} {m-2\choose i}{i\choose 3}.\]

Hence,
\begin{equation*}
\begin{split}
2^{m-5}{m-2\choose 3} +\sum_{i=3}^{m-1}{m-2\choose i-1}=\sum_{i=3}^{m-1} {m-2\choose i}{i\choose 3}+ \sum_{i=3}^{m-1}{m-2\choose i-1}{i\choose 3}\\= \sum_{i=3}^{m-1} {m-1\choose i}{i\choose 3} = 2^{m-4}{m-1\choose 3}.
\end{split}
\end{equation*}
This finishes the proof. \end{proof}

\begin{prop}\label{comb_formular_2} For $m\geq 6$, \[\sum_{i=4}^{m-2}2^{i-4}{i\choose 4}+\sum_{k=3}^{m-2}(\sum_{i=3}^{k} (m-k-1)\cdot {k-1\choose {i-1}} {i\choose 3})=2^{m-5}{m-1\choose 4}.\]
\end{prop}

\begin{proof} The equation holds for $m=6$. Assume that the equation holds for $m-1$ for $m\geq 7$, i.e.

\[\sum_{i=4}^{m-3}2^{i-4}{i\choose 4}+\sum_{k=3}^{m-3}(\sum_{i=3}^{k} (m-k-2)\cdot {k-1\choose {i-1}} {i\choose 3})=2^{m-6}{m-2\choose 4}.\]

Then
\begin{equation*}
\begin{split}
\sum_{i=4}^{m-2}2^{i-4}{i\choose 4}+\sum_{k=3}^{m-2}(\sum_{i=3}^{k} (m-k-1)\cdot {k-1\choose {i-1}} {i\choose 3})\\= 2^{m-5}{m-2\choose 4}+\sum_{i=3}^{m-2} {k-1\choose {i-1}} {i\choose 3}+\sum_{k=3}^{m-2}(\sum_{i=3}^{k} \cdot {k-1\choose {i-1}} {i\choose 3})
\end{split}
\end{equation*}

By Proposition~\ref{comb_formular_1},  \[\sum_{i=3}^{m-2} {k-1\choose {i-1}} {i\choose 3}+\sum_{k=3}^{m-3}(\sum_{i=3}^{k} \cdot {k-1\choose {i-1}}=2^{m-5}{m-2\choose 3}.\]

Also note that  \[2^{m-5}{m-2\choose 4}+2^{m-5}{m-2\choose 3}=2^{m-5}{m-1\choose 4}.\]

By induction, the equation holds for all $m\geq 6$.
\end{proof}

\end{document}